\documentclass[a4paper,10pt]{amsart}

\usepackage{hyperref}

\usepackage{amsfonts}
\usepackage{amsmath}
\usepackage{amssymb,amsxtra,amscd,amsthm}
\usepackage{mathrsfs}
\usepackage{mathtools}

\newtheorem{thm}{Theorem}

\newtheorem{corollary}[thm]{Corollary}
\newtheorem{proposition}[thm]{Proposition}

\theoremstyle{definition}
\newtheorem{definition}[thm]{Definition}
\newtheorem{rmk}{Remark}

\newcommand{\R}{\mathbb{R}}
\newcommand{\C}{\mathbb{C}}
\newcommand{\N}{\mathbb{N}}

\newcommand{\supp}{\mbox{supp}\,}
\newcommand{\dist}{\mbox{dist}\,}

\allowdisplaybreaks

\title[An approximation theorem of Runge type]{An approximation theorem of Runge type for kernels of certain non-elliptic partial differential operators}

\author{Thomas Kalmes}
\address{Chemnitz University of Technology, Faculty of Mathematics, 09107 Chemnitz, Germany}
\email{thomas.kalmes@math.tu-chemnitz.de}

\thanks{This is a preprint version of T.\ Kalmes, An approximation theorem of Runge type for kernels of certain non-elliptic partial differential operators, Bull.\ Sci.\ math.\ (2021), \href{https://doi.org/10.1016/j.bulsci.2021.103012}{https://doi.org/10.1016/j.bulsci.2021.103012}}

\begin{document}

\begin{abstract}
	For a constant coefficient partial differential operator $P(D)$ with a single characteristic direction, such as the time-dependent free Schr\"odinger operator as well as non-degenerate parabolic differential operators like the heat operator, we characterize when open subsets $X_1\subseteq X_2$ of $\R^d$ form a $P$-Runge pair. The presented condition does not require any kind of regularity of the boundaries of $X_1$ nor $X_2$. As part of our result we prove that for a large class of non-elliptic operators $P(D)$ there are smooth solutions $u$ of the equation $P(D)u=0$ on $\R^d$ with support contained in an arbitarily narrow slab bounded by two parallel characteristic hyperplanes for $P(D)$.\\
	
	
\end{abstract}

\maketitle

\section{Introduction}

From Runge's classical theorem on rational approximation it follows that for open subsets $X_1\subseteq X_2$ of the complex plane $\C$ every function holomorphic in $X_1$ can be approximated uniformly on compact subsets of $X_1$ by functions which are holomorphic in $X_2$ if and only if $\C\backslash X_1$ has no compact connected component which is contained in $X_2$. This approximation theorem has been generalized independently by Lax \cite{Lax1956} and Malgrange \cite{Malgrange1955} from holomorphic functions, i.e.\ functions in the kernel of the Cauchy-Riemann operator, to kernels of elliptic constant coefficient partial differential operators $P(D)$ and has been generalized further to kernels of elliptic differential operators with variable coefficients by Browder \cite{Browder1962}. Since then, very little improvement has been achieved in generalizing these approximation results to kernels of non-elliptic linear partial differential operators. The analogue approximation problem for the kernel of the heat operator with open subsets $X_1$ and $X_2$ of $\R^d$ has been investigated by Jones for the special case of $X_2=\R^d$ \cite{Jones1975} and by Diaz \cite{Diaz1980} for arbitrary $X_2$. However, as noted in \cite[page 359]{GaTa08} the proof of the result in \cite{Diaz1980} contains a gap. Recently, the sufficiency part of Jones' result, i.e.\ the case of $X_2=\R^d$, has been generalized to parabolic differential operators of second order with suitable variable coefficients in \cite{EnGFPS19} where also a quantitative approximation result has been obtained for the heat equation, together with some applications.

The aim of the present paper is to give an approximation result of Runge type for kernels of constant coefficient linear partial differential operators $P(D)$ with a single characteristic direction, i.e.\ the real zeros of the principal part of the polynomial $P$ form a one dimensional subspace of $\R^d$. This class of differential operators includes, among others, the time dependent free Schr\"odinger operator as well as non-degenerate parabolic operators like the heat operator. Since the class of considered partial differential operators does not consist only of hypoelliptic operators, we consider the approximation problem for the kernels of the differential operator both in the space of smooth functions as well as in the space of distributions which are equipped with the topology of local uniform convergence of all partial derivatives and with the strong dual topology, respectively. For $P(D)$ with a single characteristic direction, we present a sufficient condition for approximability, both in the smooth setting as well as in the distributional setting, for open subsets $X_1\subseteq X_2$ for which $P(D)$ is surjective both on $C^\infty(X_1)$ and $C^\infty(X_2)$, see Theorems \ref{sufficiency for P-Runge pairs} - it should be noted that for such operators $P(D)$ a geometric characterization of those open subsets $X\subseteq\R^d$ for which $P(D)$ is surjective on $C^\infty(X)$ was recently given in \cite{Kalmes19-1} so that our result can be easily evaluated. Moreover, for a large class of non-elliptic differential operators we show in Theorem \ref{necessary condition} that the presented sufficient condition is also necessary for approximability in the distributional setting. As a consequence, for certain differential operators with a single characteristic direction, we obtain on the one hand a complete geometric characterization of the open subsets $X\subseteq\R^d$ for which distributional zero solutions on $X$ can be approximated by global zero solutions (Corollary \ref{global approximation}) and on the other hand, for the particular case of non-degenerate parabolic operators, we explicitly derive in Corollary \ref{non-degenerate parabolic} a characterization of approximability for tubular domains which are the most natural domains with respect to concrete applications and for which the evaluation of the condition is particularly satisfactory. Moreover, in order to prove the necessity of our condition, we prove for a large class of non-elliptic differential operators $P(D)$ the existence of a non-trivial $u\in C^\infty(\R^d)$ satisfying $P(D)u=0$ in $\R^d$ and with support contained in an arbitrarily narrow slab bounded by two parallel characteristic hyperplanes, see Theorem \ref{zero solution}. The latter is achieved by applying an idea of Langenbruch from \cite{Langenbruch1983} to the function $u\in C^\infty(\R^d)$ satisfying $P(D)u=0$ with support equal to a characteristic half space constructed by H\"ormander in \cite[Theorem 8.6.7]{HoermanderPDO1}.

The article is organized as follows. In section \ref{statement of the main results} we formulate our main results after presenting the framework in which we consider the approximation problem. Section \ref{proof sufficiency for P-Runge pairs} is devoted to the technical proof of a sufficient condition for approximability for pairs of sets $X_1$ and $X_2$ for which $P(D)$ is surjective on $C^\infty(X_1)$ and $C^\infty(X_2)$ and when $P(D)$ has a single characteristic direction. In section \ref{proof zero solution} we prove under suitable hypothesis on $P(D)$ the existence of a non-trivial smooth function $u$ satisfying $P(D)u=0$ as well as the above stated support condition. This is done after we have collected some results about the explicit solution to the homogeneous Cauchy problem for $P(D)$ on a non-characteristic hyperplane in section \ref{prelim}. Finally, in section \ref{proof necessity} we will then provide the proofs of the necessity of the condition together with the characterization of pairs of tubular domains admitting the approximation result for non-degenerate parabolic operators as well as the characterization of those $X$ allowing distributional approximability by global solutions.

Throughout, we use standard notation from the theory of partial differential operators, see e.g.\ \cite{HoermanderPDO1}, \cite{HoermanderPDO2}, and functional analysis, see e.g.\ \cite{MeVo1997}.

\section{Statement of the main results}\label{statement of the main results}

Since the proofs of our results are quite technical, we prefer to state the main results of this article in this section which may make the reader more willing to study their proofs. Throughout the paper, $P$ denotes a non-constant polynomial with complex coefficients in $d\geq 2$ variables of degree $m$. Recall that a hyperplane $H=\{x\in\R^d;\,\langle x,N\rangle=c\}$ in $\R^d$, where $N\in\R^d\backslash\{0\}$ and $c\in\R$, is called characteristic for $P$ if $P_m(N)=0$. We then call $\mbox{span}\{N\}$ a \textit{characteristic direction for $P$}. Moreover, we assume further that there is $\xi\in\R^d$ orthogonal to $N$ which is not characteristic for $P$. Without loss of generality we assume $N=e_1=(\delta_{k,1})_{1\leq k\leq d}\in\R^d$ (Kronecker $\delta$) and $\xi=e_d$ and we write
$$P(x_1,\ldots,x_d)=\sum_{k=0}^m Q_k(x_1,\ldots,x_{d-1})x_d^k$$
with $Q_k\in\C[X_1,\ldots,X_{d-1}],0\le k\leq m$. Since $P$ is of degree $m$, it follows that the degree of $Q_k$ is bounded by $m-k$ and since $e_d$ is supposed to be non-characteristic we have $Q_m=c\in\C\backslash\{0\}$.

For an open subset $X\subseteq\R^d$ we define
$$\mathscr{E}_P(X):=\{u\in C^\infty(X);\,P(D)u=0\mbox{ in }X\}$$
and
$$\mathscr{D}'_P(X):=\{u\in\mathscr{D}'(X);\,P(D)u=0\mbox{ in }X\},$$
where for $P(x)=\sum_{|\alpha|\leq m}a_\alpha x^\alpha$ we set as usual $P(D)u=\sum_{|\alpha|\leq m}a_\alpha(-i)^{|\alpha|}\partial^\alpha u$, $u\in \mathscr{D}'(X)$, and we denote by $P_m(x):=\sum_{|\alpha|=m}a_\alpha x^\alpha$, resp.\ by $P_m(D)$, the principal part of $P$, resp.\ of $P(D)$.

We equip $C^\infty(X)$ with its usual Fr\'echet space topology, i.e.\ the topology of uniform convergence on compact subsets of $X$ of all partial derivatives which is induced by the family of seminorms $\{\|\cdot\|_{K,l};\,K\subseteq X\mbox{ compact}, l\in\N_0\}$
$$\forall\, K\subseteq X\mbox{ compact}, l\in\N_0, u\in C^\infty(X):\,\|u\|_{K,l}:=\max_{x\in K}\max_{|\alpha|\leq l}|\partial^\alpha u(x)|$$
and we denote by $\mathscr{E}(X)$ the space $C^\infty(X)$ equipped with this Fr\'echet space topology. Then $P(D)$ is a continuous linear self mapping on $\mathscr{E}(X)$ and thus, as a closed subspace of $\mathscr{E}(X)$, the space $\mathscr{E}_P(X)$ is a Fr\'echet space. Moreover, we endow $\mathscr{D}'_P(X)$ with the relative topology of $\mathscr{D}'(X)$ which is equipped with the strong dual topology as the topological dual of $\mathscr{D}(X)$.

For hypoelliptic polynomials $P$ - by definition - for every open $X\subseteq\R^d$ the spaces $\mathscr{E}_P(X)$ and $\mathscr{D}'_P(X)$
coincide algebraically (that is, every distribution $u$ on $X$ which satisfies $P(D)u=0$ in $X$ is already a smooth function). By a result of Malgrange the spaces $\mathscr{E}_P(X)$ and $\mathscr{D}'_P(X)$ also coincide as locally convex spaces. This implies in particular, that for hypoelliptic polynomials the topology on $\mathscr{E}_P(X)$ coincides with the compact-open topology, i.e.\ the topology of uniform convergence on compact subsets of $X$. For further results about linear topological properties of $\mathscr{D}'_P(X)$ for arbitrary $P$ we refer the reader to \cite{Wengenroth14}. Moreover, for the special case of a hypoelliptic $P$ with a single characteristic direction, it has recently be shown in \cite[Theorem 18]{Kalmes19-1} that, contrary to arbitrary hypoelliptic operators \cite{Kalmes12-2}, the kernel of such operators over sets which are $P$-convex for supports automatically have property $(\Omega)$, so that these kernels are (topologically isomorphic to) a quotient spaces of $s$, the space of rapidly decreasing sequences \cite[Proposition 31.6]{MeVo1997}.

A pair of open subsets $X_1\subseteq X_2$ of $\R^d$ is called a $P$\textit{-Runge pair} if the (continuous linear) restriction map
$$r_\mathscr{E}:\mathscr{E}_P(X_2)\rightarrow \mathscr{E}_P(X_1), u\mapsto u_{|X_1}$$ 
has dense range. Since elliptic polynomials (i.e.\ $P_m(x)\neq 0$ for all $x\in\R^d\backslash\{0\}$) are hypoelliptic, the Lax-Malgrange Theorem mentioned in the introduction may then be rephrased as follows: $X_1\subseteq X_2$ are a $P$-Runge pair if and only if $\R^d\backslash X_1$ has no compact connected component which is contained in $X_2$. Moreover, the particular example of the Cauchy-Riemann operator $P(D)=\frac{1}{2}(\partial_1+i\partial_2)$ gives as $\mathscr{E}_P(X)$ the space of holomorphic functions $H(X)$ equipped with the compact open topology over $X\subseteq\C$ open.

Recall that $P(D)$ is surjective on $\mathscr{E}(X)$ if and only if $X$ is $P$-convex for supports (see \cite[Section 10.6]{HoermanderPDO2}), i.e.\ if and only if
$$\forall\,\varphi\in\mathscr{D}(X):\,\dist(\supp\varphi, \R^d\backslash X)=\dist(\supp\check{P}(D)\varphi,\R^d\backslash X),$$
where $\check{P}(\xi):=P(-\xi)$ and where $\mbox{dist}$ refers to the euclidean distance. It is well known that for elliptic $P$ every open subset $X$ of $\R^d$ is $P$-convex for supports, see e.g.\ \cite[Corollary 10.8.2]{HoermanderPDO2}. While in general $P$-convexity for supports is necessary for surjectivity of $P(D)$ on $\mathscr{D}'(X)$ (see e.g.\ \cite[Theorem 10.6.6]{HoermanderPDO2}) it is not sufficient. However, for $d=2$ $P$-convexity for supports of $X\subseteq\R^2$ already implies surjectivity on $\mathscr{D}'(X)$, as was recently shown in \cite{Kalmes11} (see also \cite{Kalmes10}). 

Our first result gives a sufficient condition for the $P$-Runge pairs consisting of sets which are $P$-convex for supports for polynomials $P$ with a single characteristic direction. 

\begin{thm}\label{sufficiency for P-Runge pairs}
	Let $P\in\C[X_1,\ldots,X_d], d\geq 2,$ be of degree $m$ such that $\{x\in\R^d;\,P_m(x)=0\}=\mbox{span}\{e_1\}$. Moreover, let  $X_1\subseteq X_2\subseteq \R^d$ be open and $P$-convex for supports. Then, among the following, $i)$ and $ii)$ are equivalent and follow from $iii)$, where
	\begin{enumerate}
		\item[i)] $X_1$ and $X_2$ is a $P$-Runge pair.
		\item[ii)] The restriction map $r_\mathscr{D'}:\mathscr{D}'_P(X_2)\rightarrow\mathscr{D}'_P(X_1), u\mapsto u_{|X_1}$ has dense range.
		\item[iii)] There is no characteristic hyperplane $H_c=\{x\in\R^d;\,x_1=c\}, c\in\R$, for $P(D)$ such that $X_2$ contains a compact connected component of $(\R^d\backslash X_1)\cap H_c$.
	\end{enumerate}
\end{thm}

The proof of Theorem \ref{sufficiency for P-Runge pairs} is given in Section \ref{proof sufficiency for P-Runge pairs} below. A geometric characterization of $P$-convexity for supports for polynomials with a single characteristic direction was recently given in \cite{Kalmes19-1}, see also Theorem \ref{characterization p-convexity} below.

The sufficient condition for $P$-Runge pairs from Theorem \ref{sufficiency for P-Runge pairs} turns out to be necessary for the restriction map $r_{\mathscr{D}'}:\mathscr{D}'(X_2)\rightarrow\mathscr{D}'(X_1)$ to have dense range for arbitrary open subsets $X_1\subseteq X_2\subseteq\R^d$ and non-elliptic $P$ of degree $m$ under a mild additional property. More precisely, we have the following theorem whose proof is given in Section \ref{proof necessity}

\begin{thm}\label{necessary condition}
	Let $P\in\C[X_1,\ldots,X_d], d\geq 2$, be of degree $m$ such that $e_1$ is characteristic for $P$ while $e_d$ is not. Moreover, assume that
	$$\forall\,x\in\R^d:\,P(x_1,\ldots,x_d)=\sum_{k=0}^m Q_k(x_1,\ldots,x_{d-1})x_d^k,$$
	with $Q_k\in\C[X_1,\ldots,X_{d-1}]$ satisfying $\deg_{x_1}(Q_k)<m-k$ for every $0\leq k\leq m-1$, where $\deg_{x_1}Q_k$ denotes the degree of the $x_1$-variable of $Q_k$.
	
	Let $X_1\subseteq X_2\subseteq\R^d$ be open such that the restriction map $$r_{\mathscr{D}'}:\mathscr{D}'_P(X_2)\rightarrow\mathscr{D}'_P(X_1), u\mapsto u_{|X_1}$$
	has dense range. Then there is no characteristic hyperplane $H_c=\{x\in\R^d;\, x_1=c\}, c\in\R$, of $P(D)$ such that $X_2$ contains a compact connected component of $(\R^d\backslash X_1)\cap H_c$.
\end{thm}

Clearly, combining Theorems \ref{sufficiency for P-Runge pairs} and \ref{necessary condition}, for polynomials $P$ with single characteristic direction $\mbox{span}\{e_1\}$ which also satisfy the mild additional hypothesis of the latter theorem, we obtain a geometric characterization of the $P$-Runge pairs of subsets $X_1\subseteq X_2\subseteq\R^d$ which are both $P$-convex for supports. We leave it to the reader to explicitly state this combination of Theorems \ref{sufficiency for P-Runge pairs} and \ref{necessary condition}. Instead, we prefer to illustrate the applicability of our results for two particular cases. In the first case, we consider approximability by solutions defined on the whole space, and the second application is to the time dependent free Schr\"odinger operator $P(D)=i\frac{\partial}{\partial t}+\Delta_x$ as well as non-degenerate parabolic operators like the heat operator $P(D)=\frac{\partial}{\partial t}-\Delta_x$ on tubular domains which from the point of view of concrete problems are the most important ones and for which the evaluation of the geometric condition stated in Theorem \ref{sufficiency for P-Runge pairs} is particularly nice. As usual, in this context we denote elements of $\R^{n+1}$ by $(t,x)$ with $t\in\R$ and $x\in\R^n$. (Thus, $d=n+1$ with $n$ spatial variables and one time variable.) Recall that for an elliptic polynomial $Q\in\C[X_1,\ldots,X_n]$ of degree $m>1$ with real coefficients in its principal part, the polynomial $P(t,x) := (-it)^r-Q(x), r\in\N$, $r<m$ odd, is semi-elliptic so that the operator $P(D)=\partial_t^r-Q(D_x)$ is hypoelliptic (see e.g. \cite[Theorem 11.1.11]{HoermanderPDO2}). The special case $r=1$ yields a non-degenerate parabolic operator.

For the special case of the heat operator, the next result is Jones' Approximation Theorem \cite{Jones1975} cited in the introduction. Clearly, for hypoelliptic $P$ we obtain a characterization of the open subsets $X\subseteq\R^d$ for which $(X,\R^d)$ is a $P$-Runge pair. 

\begin{corollary}\label{global approximation}
	Let $P\in\C[X_1,\ldots,X_d], d\geq 2$, be a polynomial of degree $m$ such that $e_1$ is characteristic for $P$ while $e_d$ is not. Moreover, assume that
	$$\forall\,x\in\R^d:\,P(x_1,\ldots,x_d)=\sum_{k=0}^m Q_k(x_1,\ldots,x_{d-1})x_d^k,$$
	with $Q_k\in\C[X_1,\ldots,X_{d-1}]$ satisfying $\deg_{x_1}(Q_k)<m-k$ for every $0\leq k\leq m-1$. Then, for an open subset $X\subseteq\R^d$, the following conditions are equivalent and imply that $(X,\R^d)$ is a $P$-Runge pair.
	\begin{itemize}
		\item[i)] $r_{\mathscr{D}'}:\mathscr{D}'_P(\R^d)\rightarrow\mathscr{D}'_P(X)$ has dense range.
		\item[ii)] There is no $c\in\R$ such that with $H_c=\{x\in\R^d;\,x_1=c\}$ the set $(\R^d\backslash X)\cap H_c$ has a compact connected component.
	\end{itemize}
\end{corollary}

The proof of the previous corollary will be given in Section \ref{proof necessity} as will be the proof of the next one.

\begin{corollary}\label{non-degenerate parabolic}
	Let $I_1\subseteq I_2\subseteq\R$ and $X_1\subseteq X_2\subseteq\R^n$ be open. Then the following hold.
	\begin{itemize}
		\item[i)] Let $Q\in\C[X_1,\ldots,X_n]$ be an elliptic polynomial of degree $m>1$ with real coefficients in its principal part and let $r\in\N$ be odd with $r<m$. Every solution $u$ of the partial differential equation
		$$\partial_t^r u-Q(D_x)u=0\mbox{ in }I_1\times X_1$$
		is the local uniform limit in $I_1\times X_1$ of a sequence of solutions of the same differential equation in $I_2\times X_2$ if and only if $\R^n\backslash X_1$ does not have a compact connected component contained in $X_2$.
		\item[ii)] Every smooth solution, resp.\ distributional solution $u$ of the time-dependent free Schr\"odinger equation
		$$i\,\partial_t u+\Delta_x u=0\mbox{ in }I_1\times X_1$$
		is the limit in $\mathscr{E}(I_1\times X_1)$, resp.\ in $\mathscr{D}'(I_1\times X_1)$, of a sequence of solutions in $\mathscr{E}(I_2\times X_2)$, resp.\ a net of solutions in $\mathscr{D}'(I_2\times X_2)$, of the same differential equation in $I_2\times X_2$ if and only if $\R^n\backslash X_1$ does not have a compact connected component contained in $X_2$.
	\end{itemize}
\end{corollary}

For an application of the above characterization of $P$-Runge pairs, see \cite[Section 4]{Kalmes19-2}.\\

The proof of Theorem \ref{necessary condition} uses the following result which is of independent interest. As is well-known, see e.g.\ \cite[Theorem 8.6.7]{HoermanderPDO1}, if $P$ is a polynomial with characteristic vector $e_1$, there is $u\in\mathscr{E}_P(\R^d)$ with $\supp u=\{x\in\R^d;\, x_1\leq 0\}$. Under the mild additional assumptions on $P$ from Theorem \ref{necessary condition} we show that there is $u\in\mathscr{E}_P(\R^d)\backslash\{0\}$ whose support is bounded with respect to $x_1$. More precisely, the following is true.

\begin{thm}\label{zero solution}
	Let $P\in\C[X_1,\ldots,X_d], d\geq 2$, be of degree $m$ and such that $e_1$ is characteristic for $P$ while $e_d$ is not. Assume that
	$$P(x_1,\ldots,x_d)=\sum_{j=0}^m Q_j(x_1,\ldots,x_{d-1})x_d^j$$
	with $Q_k\in\C[X_1,\ldots,X_{d-1}]$ such that $\deg_{x_1}(Q_k)< m-k$ for all $0\leq k\leq  m-1$. Then, for every $a>0$ and $\varepsilon\in (0,a)$ there is $u\in \mathscr{E}_P(\R^d)$ such that
	$$[-a,-\varepsilon]\times\R^{d-1}\subseteq\supp u\subseteq [-(a+\varepsilon),0]\times\R^{d-1}$$
	and the restriction of $u$ to $(-a,-\varepsilon)\times\R^{d-1}$ is real analytic.
\end{thm}

The proof of the above theorem is given in Section \ref{proof zero solution}. Our strategy of the proof is to combine an idea due to Langenbruch from \cite{Langenbruch1983}, where for certain partial differential operators fundamental solutions with partially bounded supports have been constructed, with H\"ormander's construction of a function $v\in \mathscr{E}(\R^d)$ which satisfies $P(D)v=0$ and $\supp v=\{x\in\R^d;\, x_1\leq 0\}$. This approach uses a power series ansatz to solve the homogeneous Cauchy problem $P(D)u=0$ with Cauchy data on the non-characteristic hyperplane $\{x\in\R^d;\, x_d=0\}$, $D_d^j u(x_1,\ldots,x_{d-1},0)=g(x_1)D_d^jv(x_1,\ldots,x_{d-1},0), 0\leq j\leq m-1,$ for a suitable cut-off function $g$. We collect the necessary results on the Cauchy problem we shall employ to prove Theorem \ref{zero solution} in section \ref{prelim}.

\section{Proofs of Theorem \ref{sufficiency for P-Runge pairs}}\label{proof sufficiency for P-Runge pairs}

In this section we prove Theorem \ref{sufficiency for P-Runge pairs}. The proof uses the following general approximation result for kernels of differential operators. The equivalence of ii) and iv) is due to Tr\`eves \cite[Theorem 26.1]{Treves1967-2} which should be compared to a result due to Malgrange \cite{Malgrange1955} (see also \cite[Theorem 10.5.2]{HoermanderPDO2}). A generalization of this equivalence to the ultradifferentiable setting has been achieved by Wiechert \cite[Satz 15]{Wiechert1982}.

\begin{thm}\label{generalizing Treves}
	Let $P\in\C[X_1,\ldots,X_d]\backslash\{0\}$ and let $X_1\subseteq X_2\subseteq\R^d$ be open sets such that $X_2$ is $P$-convex for supports. Then the following are equivalent.
	\begin{itemize}
		\item[i)] $X_1$ is $P$-convex for supports and the restriction map
		$$r_\mathscr{D'}:\mathscr{D}'_P(X_2)\rightarrow\mathscr{D}'_P(X_1),u\mapsto u_{|X_1}$$
		has dense range.
		\item[ii)] $X_1$ is $P$-convex for supports and the restriction map
		$$r_{\mathscr{E}}:\mathscr{E}_P(X_2)\rightarrow\mathscr{E}_P(X_1),f\mapsto f_{|X_1}$$
		has dense range, i.e.\ $X_1$ and $X_2$ form a $P$-Runge pair.
		\item[iii)] For every $u\in\mathscr{E}'(X_2)$ with $\supp \check{P}(D)u\subseteq X_1$ it holds $\supp u\subseteq X_1$.
		\item[iv)] For every $\varphi\in\mathscr{D}(X_2)$ with $\supp \check{P}(D)\varphi\subseteq X_1$ it holds $\supp \varphi\subseteq X_1$.
	\end{itemize}
\end{thm}

As does the proof of Wiechert's generalization to Tr\`eves' Theorem in  \cite{Wiechert1982}, the proof of one implication not covered by Tr\`eves' Theorem will be based on Grothendieck-K\"othe duality, see \cite{Grothendieck1953}. We recall some facts of this theory for the reader's convenience. 

For $X\subseteq\R^d$ open and $P$-convex for supports, the topological dual space of $\mathscr{E}_P(X)$ is isomorphic to a space of certain distributional solutions $u$ of the equation $\check{P}(D)u=0$ outside a compact subset of $X$ which may depend on $u$. More precisely, recall that for a compact $K\subseteq\R^d$ an essential extension of $u\in\mathscr{D}'(\R^d\backslash K)$ is a distribution $U\in\mathscr{D}'(\R^d)$ for which $u_{|\R^d\backslash L}=U_{|\R^d\backslash L}$ for some compact $L\supseteq K$. Multiplying $u$ with a smooth function having support in $\R^d\backslash K$ which is equal to $1$ outside a compact superset of $K$ shows that every $u\in\mathscr{D}'(\R^d\backslash K)$ has essential extensions. For $u\in\mathscr{D}'_{\check{P}}(\R^d\backslash K)$ and any essential extension $U$ we have $\check{P}(D)U\in\mathscr{E}'(\R^d)$ so that $E*\check{P}(D)U$ is defined, where $E$ is a fixed fundamental solution of $\check{P}(D)$. Then $u\in\mathscr{D}'_{\check{P}}(\R^d\backslash K)$ is called regular at infinity (with respect to $E$) if for one (and then every) of its essential extensions $U$ it holds $E*\check{P}(D)U=U$. We set
\begin{align*}
	R\mathscr{D}'_{\check{P}}(\R^d\backslash K)&:=\{u\in\mathscr{D}'_{\check{P}}(\R^d\backslash K):\,u\mbox{ regular at infinity with respect to }E\},\\
	R\mathscr{E}_{\check{P}}(\R^d\backslash K)&:=R\mathscr{D}'_{\check{P}}(\R^d\backslash K)\cap \mathscr{E}(\R^d\backslash K).	
\end{align*}
Then $R\mathscr{D}'_{\check{P}}(\R^d\backslash K)$ and $R\mathscr{E}_{\check{P}}(\R^d\backslash K)$ are closed subspaces of $\mathscr{D}'(\R^d\backslash K)$ and of $\mathscr{E}(\R^d\backslash K)$, respectively. With these spaces we define
\begin{align*}
	R\mathscr{D}'_{\check{P}}(X^c)&:=\cup_{K\subseteq X\mbox{ compact}}R\mathscr{D}'_{\check{P}}(\R^d\backslash K)=\varinjlim_{K\subseteq X\mbox{ compact}}R\mathscr{D}'_{\check{P}}(\R^d\backslash K),\\
	R\mathscr{E}_{\check{P}}(X^c)&:=\cup_{K\subseteq X\mbox{ compact}}R\mathscr{E}_{\check{P}}(\R^d\backslash K)=\varinjlim_{K\subseteq X\mbox{ compact}}R\mathscr{E}_{\check{P}}(\R^d\backslash K).
\end{align*}
Then, in case $X$ is $P$-convex for supports, it follows that $R\mathscr{D}'_{\check{P}}(X^c)$ equipped with the inductive limit topology and the dual space $\mathscr{E}_P(X)'$ of $\mathscr{E}_P(X)$ equipped with the strong topology are topologically isomorphic via
\begin{equation}\label{Grothendieck-Koethe isomorphism}
	\Phi:R\mathscr{D}'_{\check{P}}(X^c)\rightarrow \mathscr{E}_P(X)',\langle\Phi(u), f\rangle:=\langle \check{P}(D)(\psi u),f\rangle,
\end{equation}
where the duality bracket on the right hand side denotes the usual duality between $\mathscr{E}'(X)$ and $\mathscr{E}(X)$. Here for $u\in\mathscr{D}'_{\check{P}}(\R^d\backslash K)$ the function $\psi\in \mathscr{E}(\R^d)$ is arbitrary as long as $\psi$ vanishes in a neighborhood of $K$ and is equal to 1 outside a compact subset $L\subseteq X$. It should be noted that $\check{P}(D)(\psi u)$ has compact support but usually $\psi u$ does not.

Moreover, again in case $X$ is $P$-convex for supports, equipping $R\mathscr{E}_{\check{P}}(X^c)$ with the inductive limit topology and the dual space $\mathscr{D}'_P(X)'$ of $\mathscr{D}'_P(X)$ with the strong topology
$$\Psi: R\mathscr{E}_{\check{P}}(X^c)\rightarrow\mathscr{D}'_P(X)',\langle\Psi(f),u\rangle:=\langle\check{P}(D)(\psi f),u\rangle$$
is a topological isomorphism, where the duality bracket on the right hand side denotes the usual duality between $\mathscr{D}(X)$ and $\mathscr{D}'(X)$ and where again $\psi\in \mathscr{E}(\R^d)$ is as above.\\

\begin{proof}[Proof of Theorem \ref{generalizing Treves}]
	For the equivalence of ii) and iii), see \cite[Theorem 26.1]{Treves1967-2}. We first show that i) implies iv). Fix $\varphi\in\mathscr{D}(X_2)$ with $\supp \check{P}(D)\varphi\subseteq X_1$, i.e.\ $\check{P}(D)\varphi\in\mathscr{D}(X_1)$. For every $u\in \mathscr{D}'_P(X_2)$ we have
	$$\langle \check{P}(D)\varphi,r_\mathscr{D'}(u)\rangle=\langle \check{P}(D)\varphi, u\rangle=\langle\varphi,P(D)u\rangle=0.$$
	Thus, $\check{P}(D)\varphi$ vanishes on the range of $r_\mathscr{D'}$ which is dense in $\mathscr{D}'_P(X_1)$ by hypothesis, so that $\check{P}(D)\varphi_{|\mathscr{D}'_P(X_1)}=0$. Since $X_1$ is $P$-convex for supports,  a result of Floret \cite[page 232]{Floret1980} ensures that $\check{P}(D)\big(\mathscr{D}(X_1)\big)$ is closed in $\mathscr{D}(X_1)$. Therefore, for the polar of $\mathscr{D}'_P(X_1)$ with respect to the dual pair $(\mathscr{D}(X_1),\mathscr{D}'(X_1))$ we have
	$$\mathscr{D}'_P(X_1)^\circ=\overline{\check{P}(D)\big(\mathscr{D}(X_1)\big)}^{\mathscr{D}(X_1)}=\check{P}(D)\big(\mathscr{D}(X_1)\big).$$
	Since the dual space of $\mathscr{D}'_P(X_1)$ is canonically isomorphic to the quotient space
	$$\mathscr{D}(X_1)/\mathscr{D}'_P(X_1)^\circ=\mathscr{D}(X_1)/\check{P}(D)\big(\mathscr{D}(X_1)\big),$$
	$\check{P}(D)\varphi_{|\mathscr{D}'_P(X_1)}=0$ implies $\check{P}(D)\varphi\in \check{P}(D)(\mathscr{D}(X_1))$, i.e.\ there is $u\in\mathscr{D}(X_1)$ with $\check{P}(D)\varphi=\check{P}(D)u$. Since $\check{P}(D)$ is injective on $\mathscr{D}(\R^d)$ we conclude $\varphi=u$ and thus $\supp\varphi\subseteq X_1$.
	
	Next, we show that iv) implies iii). Let $u\in\mathscr{E}'(X_2)$ with $\check{P}(D)u\subseteq X_1$. Let $\phi\in\mathscr{D}(\R^d)$ be non-negative with support contained in the open unit ball about the origin and $\int\phi(x) dx=1$. With $\phi_\varepsilon(x):=\varepsilon^{-d}\phi(x/\varepsilon)$ it holds that $v*\phi_\varepsilon\rightarrow v$ in $\mathscr{E}'(\R^d)$ as $\varepsilon\rightarrow 0$ with $v*\phi_\varepsilon\in\mathscr{D}\big(\supp v+B(0,\varepsilon)\big)$, $v\in\mathscr{E}'(\R^d)$. Since $\check{P}(D)\big(u*\phi_\varepsilon\big)=\check{P}(D)u*\phi_\varepsilon$, iii) follows from iv). Trivially, iii) implies iv).
	
	In order to finish the proof, it remains to show that ii) implies i). If ii) holds we only have to show that the transposed of the restriction $r_\mathscr{D'}$ is injective. Since $X_1$ and $X_2$ are $P$-convex for supports, due to the Grothendieck-K\"othe duality, this is equivalent to the injectivity of the inclusion
	$$j_{\mathscr{E}}:R\mathscr{E}_{\check{P}}(X_1^c)\hookrightarrow R\mathscr{E}_{\check{P}}(X_2^c), f\mapsto f.$$
	But ii) implies the injectivity of
	$$j_\mathscr{D'}:R\mathscr{D}'_{\check{P}}(X_1^c)\hookrightarrow R\mathscr{D}'_{\check{P}}(X_2^c), u\mapsto u$$
	and since $j_{\mathscr{D'}|R\mathscr{E}_{\check{P}}(X_1^c)}=j_{\mathscr{E}}$, i) follows. This completes the proof.
\end{proof}

Apart from Theorem \ref{generalizing Treves} the recent geometrical characterization of $P$-convexity for supports for polynomials with a single characteristic direction obtained in \cite{Kalmes19-1} will be needed to prove the sufficiency of iii) for i) and ii) in Theorem \ref{sufficiency for P-Runge pairs}. Recall that a real valued continuous function $f$ on an open subset $X$ of $\R^d$ is said to satisfy the minimum principle in a closed set $F$ of $\R^d$ if for every compact set $K\subseteq F\cap X$ it holds
$$\min_{x\in K}f(x)=\min_{x\in\partial_F K}f(x),$$
where $\partial_F K$ is the boundary of $K$ as a subset of $F$. Combining \cite[Corollary 5]{Kalmes19-1} and \cite[Lemma 4]{Kalmes19-1} we have the following.

\begin{thm}\label{characterization p-convexity}
	Let $P\in\C[X_1,\ldots,X_d], d\geq 2,$ have a single characteristic direction. For $X\subseteq\R^d$ open let
	$$d_X:X\rightarrow\R, d_X(x):=\inf\{|x-y|;\,y\in\R^d\backslash X\}.$$
	Then the following are equivalent.
	\begin{enumerate}
		\item[i)] $X$ is $P$-convex for supports.
		\item[ii)] $d_X$ satisfies the minimum principle in every characteristic hyperplane for $P$.
		\item[iii)] For each compact subset $K\subseteq X$ and every
		$$x\in \{y\in X;\,d_X(y)<\dist(K,X^c)\}$$
		there is $\gamma:[0,\infty)\rightarrow X$ a continuous and piecewise continuously differentiable curve with $\gamma(0)=x, \gamma'(t)\in \{y\in\R^d;\,P_m(y)=0\}^\perp$ whenever $\gamma$ is differentiable in $t$, and $ \gamma([0,\infty))\cap K=\emptyset$ such that $$\liminf_{t\rightarrow\infty}\dist(\gamma(t),\partial_\infty X)=0,$$
		where $\partial_\infty X$ denotes the boundary of $X$ in the one point compactification of $\R^d$.
	\end{enumerate}
\end{thm} 

Now, we are ready to prove Theorem \ref{sufficiency for P-Runge pairs}.

\begin{proof}[Proof of Theorem \ref{sufficiency for P-Runge pairs}]
	By Theorem \ref{generalizing Treves}, i) and ii) are equivalent. Thus, it remains to show that iii) implies i) and ii). In order to do so, we will apply Theorem \ref{generalizing Treves}.
	
	We denote by $W$ the orthogonal complement in $\R^d$ of the one dimensional subspace $\{x\in\R^d;\,P_m(x)=0\}$. Then, every characteristic hyperplane for $P$ is of the form $x+W$ with $x\in\R^d$. For $X\subseteq\R^d$ we denote by $\partial_\infty X$ the boundary of $X$ in the one-point compactification of $\R^d$, thus $\infty\in\partial_\infty X$ whenever $X$ is an unbounded subset of $\R^d$. For $y,z\in\R^d$ we denote by $[y,z]$ the convex hull of $\{y,z\}$.
	
	In view of Theorem \ref{generalizing Treves} we have to show that $\supp\varphi\subseteq X_1$ for every $\varphi\in\mathscr{D}(X_2)$ with $\supp\check{P}(D)\varphi\subseteq X_1$. Thus, let $\varphi\in\mathscr{D}(X_2)$ be such that $K:=\supp\check{P}(D)\varphi\subseteq X_1$. Moreover, we fix
	$$x\in\{y\in X_1;\,\dist(y,X_1^c)<\dist(K,X_1^c)\}.$$
	We shall show that there is a continuous and piecewise continuously differentiable curve $\alpha:[0,\infty)\rightarrow X_2$ satisfying
	\begin{enumerate}
		\item[$\alpha\, 1)$] $\alpha(0)=x$,
		\item[$\alpha\, 2)$] $\alpha([0,\infty))\cap K=\emptyset$,
		\item[$\alpha\, 3)$] $\lim_{t\rightarrow\infty}\dist(\alpha(t),\partial_\infty X_2)=0$,
		\item[$\alpha\, 4)$] $\alpha'(t)\in W$ for every $t\in[0,\infty)$ where $\alpha$ is continuously differentiable.
	\end{enumerate}
	
	Before we prove that such a curve $\alpha$ exists, let us show how the theorem follows from this. Because $\supp \varphi$ is a compact subset of $X_2$ it follows from $\alpha\,3)$ that there is $T>0$ with $\alpha(T)\notin\supp \varphi$. Moreover, using $\alpha\,2)$ we can find $\varepsilon>0$ such that the open ball $B(\alpha(T),\varepsilon)$ of radius $\varepsilon$ about $\alpha(T)$ does not intersect $\supp \varphi$, $\alpha([0,T])+B(0,\varepsilon)\subseteq X_2$ and $K\cap(\alpha([0,T])+B(0,\varepsilon))=\emptyset$, where $\alpha([0,T])+B(0,\varepsilon)=\{y+z;y\in\alpha([0,T]),z\in B(0,\varepsilon)\}$.
	
	Next, we choose $0=t_0<t_1<\ldots<t_k=T$ such that for each $j=1,\ldots,k$ the restriction of $\alpha$ to $[t_{j-1},t_j]$ is continuously differentiable and
	\[\left|\int_{t_{j-1}}^{t_j}\alpha'(t)dt\right|<\frac{\varepsilon}{2}.\]
	We define
	\[f:[0,k]\rightarrow\R^d,s\mapsto\alpha\left(t_{\lfloor s\rfloor}\right)+\left(s-\lfloor s\rfloor\right)\int_{t_{\lfloor s\rfloor}}^{t_{\lfloor s\rfloor}+1}\alpha'(t)dt,\]
	where $\lfloor s\rfloor$ denotes the integer part of $s$. Then $f$ is a polygonal curve in $x+W$ by $\alpha\,1)$ and $\alpha\,4)$. Obviously, $f([j-1,j])=[\alpha(t_{j-1}),\alpha(t_j)], j=1,\ldots,k$. Moreover, due to the choice of $\varepsilon$, we have $f([0,k])+B(0,\frac{\varepsilon}{2})\subseteq X_2\backslash K$.
	
	For $N\in\{y\in\R^d;\,P_m(y)=0\}\backslash\{0\}$ and $c\in\R$ let
	\[H_{N,c}=\{y\in\R^d;\,\langle y,N\rangle=c\}\]
	be the corresponding characteristic hyperplane for $P$. Since $[\alpha(t_{k-1}),\alpha(T)]\subseteq x+W$ and $N\in W^\perp$ it follows that $H_{N,c}$ intersects $B(\alpha(T),\varepsilon)$ whenever $H_{N,c}$ intersects $[\alpha(t_{k-1}),\alpha(T)]+B(0,\varepsilon)$. By the choice of $\varepsilon$ we have $\varphi_{|B(\alpha(T),\varepsilon)}=0$ so that by \cite[Theorem 8.6.8]{HoermanderPDO1} $\varphi$ vanishes in $[\alpha(t_{k-1}),\alpha(T)]+B(0,\varepsilon)$.
	
	Repetition of this argument yields that $\varphi$ vanishes in $f([0,k])+B(0,\varepsilon)$, in particular $x=\alpha(0)=f(0)$ does not belong to $\supp \varphi$. Since $x$ was chosen arbitrarily from the set
	$$\{y\in X_1;\,\dist(y,X_1^c)<\dist(K,X_1^c)\}$$
	we conclude with the aid of the Theorem of Supports (see e.g.\ \cite[Theorem 7.3.2]{HoermanderPDO1}), which states that the convex hulls of $\supp\varphi$ and $\supp\check{P}(D)\varphi=K$ coincide,
	$$\supp\varphi\subseteq\big(\{y\in X_1;\,\dist(y,X_1^c)\geq\dist(K,X_1^c)\}\cap ch K\big)\cup\big((X_2\backslash X_1)\cap\supp\varphi\big),$$
	where $ch K$ denotes the convex hull of $K$. Setting
	$$L_1:=\{y\in X_1;\,\dist(y,X_1^c)\geq\dist(K,X_1^c)\}\cap ch K$$
	and $L_2=:(X_2\backslash X_1)\cap\supp\varphi$, $L_1$ and $L_2$ are disjoint compact subsets of $X_2$ with $L_1\subseteq X_1$. Since $\supp\varphi\subseteq L_1\cup L_2$ we can decompose $\varphi=\varphi_1+\varphi_2$ with $\varphi_1\in\mathscr{D}(X_1)$ and $\varphi_2\in\mathscr{D}(X_2\backslash X_1)$. Because $\check{P}(D)\varphi_1,\check{P}(D)\varphi\in\mathscr{D}(X_1)$ it follows
	$$\check{P}(D)\varphi_2=\check{P}(D)(\varphi-\varphi_1)\in\mathscr{D}(X_2\backslash X_1)\cap\mathscr{D}(X_1)=\{0\}$$
	which together with the injectivity of $\check{P}(D)$ on $\mathscr{D}(\R^d)$ implies $\varphi_2=0$. This finally yields $\varphi=\varphi_1\in\mathscr{D}(X_1)$ so that Theorem \ref{generalizing Treves} gives the desired result once the existence of the curve $\alpha$ is verified.
	
	We denote by $C$ the connected component of $(X_1\backslash K)\cap (x+W)$ which contains $x$. As an open subset of the pathwise connected set $x+W$ the set $C$ is locally pathwise connected (and connected) hence pathwise connected.
	
	We precede by distinguishing two cases. First, let us assume that $C$ is unbounded. Since $C$ is pathwise connected there is a continuous piecewise continuously differentiable curve $\tilde{\alpha}:[0,\infty)\rightarrow C$ such that $\tilde{\alpha}(0)=x$ and $\lim_{t\rightarrow\infty}|\tilde{\alpha}(t)|=\infty$. With this $\tilde{\alpha}$ one easily constructs a curve $\alpha$ as desired, taking into account that $C\subseteq (x+W)$ implies $\tilde{\alpha}'(t)\in W$ for every $t$ where $\tilde{\alpha}$ is differentiable.
	
	Next, let us assume that $C$ is bounded. Since $X_1$ is assumed to be $P$-convex for supports it follows from \cite[Lemma 4, Corollary 5]{Kalmes19-1}, see Theorem \ref{characterization p-convexity} above, that there is a continuous and piecewise continuously differentiable curve $\gamma:[0,\infty)\rightarrow X_1$ satisfying
	\begin{enumerate}
		\item[i)] $\gamma(0)=x$,
		\item[ii)] $\gamma([0,\infty))\cap K=\emptyset$,
		\item[iii)] $\gamma'(t)\in W$ for each $t$ where $\gamma$ is differentiable,
		\item[iv)] $\liminf_{t\rightarrow\infty}\dist(\gamma(t),\partial_\infty X_1)=0$.
	\end{enumerate} 
	From properties i)-iii) of $\gamma$ it follows that $\gamma([0,\infty))\subseteq (X_1\backslash K)\cap (x+W)$ implying $\gamma([0,\infty))\subseteq C$. Since we assumed $C$ to be bounded, property iv) of $\gamma$ yields in fact
	\begin{enumerate}
		\item[iv')] $\liminf_{t\rightarrow\infty}\dist(\gamma(t),\partial X_1)=0$.
	\end{enumerate}
	Let $\xi\in\partial X_1$ and $(t_n)_{n\in\N}$ be a strictly increasing sequence in $[0,\infty)$ tending to infinity such that $\lim_{n\rightarrow\infty}\gamma(t_n)=\xi$. From $\gamma([0,\infty))\subseteq x+W$ we conclude $\xi\in x+W$. Next, let $\varepsilon>0$ be such that $B[\xi,\varepsilon]$, the closed ball in $\R^d$ about $\xi$ with radius $\varepsilon$, does not intersect $K$. We choose $T>0$ such that $\gamma(T)\in B(\xi,\varepsilon)$, the open $\varepsilon$-ball in $\R^d$ about $\xi$, and we set
	$$I:=\{\lambda\in[0,1];\,(1-\lambda)\gamma(T)+\lambda\xi\in\partial X_1\}.$$
	Then $1\in I$ and
	$$\forall\,\lambda\in I:\,(1-\lambda)\gamma(T)+\lambda\xi\in B(\xi,\varepsilon)\cap (x+W).$$
	From $\gamma(T)\in X_1$ it follows
	$$\lambda_0:=\inf I>0.$$
	Moreover
	$$\xi_0:=(1-\lambda_0)\gamma(T)+\lambda_0\xi\in\partial X_1$$
	as well as
	$$\forall\,\lambda\in[0,\lambda_0):\,(1-\lambda)\gamma(T)+\lambda\xi\in X_1\cap B(\xi,\varepsilon)\cap(x+W)\subseteq (X_1\backslash K)\cap (x+W).$$
	Then
	\begin{align*}
		&\tilde{\gamma}:[0,T+\lambda_0]\rightarrow\big((X_1\backslash K)\cup\{\xi_0\}\big)\cap\big(x+W\big),\\
		&\tilde{\gamma}(t):=\begin{cases}
			\gamma(t), & t\leq T,\\ \big(1-(t-T)\big)\gamma(T)+(t-T)\xi_0, &t>T
		\end{cases}
	\end{align*}
	is a well-defined, continuous and piecewise continuously differentiable curve such that $\tilde{\gamma}(0)=x$, $\tilde{\gamma}(T+\lambda_0)=\xi_0\in\partial X_1$, $\tilde{\gamma}([0,T+\lambda_0))\subseteq(X_1\backslash K)\cap (x+W)$, and $\tilde{\gamma}'(t)\in W$ for every $t$ where $\tilde{\gamma}$ is differentiable. In case $\xi_0\in\partial X_2$ the curve
	$$\alpha:[0,\infty)\rightarrow(X_2\backslash K)\cap (x+W),\alpha(t):=\tilde{\gamma}\big(\frac{t}{t+1}(T+\lambda_0)\big)$$
	is as desired. In case $\xi_0\notin\partial X_2$ we denote by $C_0$ the connected component of $$(\R^d\backslash X_1)\cap(x+W)=(\R^d\backslash X_1)\cap(\xi_0+W)$$
	which contains $\xi_0$. It follows from the local pathwise connectedness of $(\R^d\backslash K)\cap(\xi_0+W)$ that $C_0$ is pathwise connected.
	
	In case $C_0$ is unbounded, there is thus a continuous and piecewise continuously differentiable curve
	$$\beta:[0,\infty)\rightarrow(\R^d\backslash K)\cap (x+W)$$
	with $\beta(0)=x$ and $\lim_{t\rightarrow\infty}|\beta(t)|=\infty$. If the range of $\beta$ does not intersect $\partial X_2$ we define $\tilde{\beta}:=\beta$ otherwise we set
	$$s:=\inf\{t>0;\,\beta(t)\in\partial X_2\}$$
	so that $s>0$ since $\xi_0\notin\partial X_2$ and
	$$\tilde{\beta}:[0,\infty)\rightarrow (X_2\backslash K)\cap (x+W),\tilde{\beta}(t):=\beta(\frac{t}{1+t}s).$$
	In both cases we define with $\tilde{\gamma}$ from above and $\tilde{\beta}$
	$$\alpha:[0,\infty)\rightarrow(X_2\backslash K)\cap(x+W),\alpha(t):=\begin{cases}
		\tilde{\gamma}(t),&t\leq T+\lambda_0\\ \tilde{\beta}(t-T-\lambda_0),&t>T+\lambda_0,
	\end{cases}$$
	which satisfies all requirements.
	
	Finally, in case $C_0$ is bounded, $C_0$ is compact. By the hypothesis on the compact connected components of the intersection of $\R^d\backslash X_1$ with characteristic hyperplanes it follows that $C_0$ intersects $\R^d\backslash X_2$. Fix $v\in C_0\cap(\R^d\backslash X_2)$. Since $C_0$ is a pathwise connected subset of $(\R^d\backslash X_1)\cap (x+W)$ there is a continuous and piecewise continuously differentiable curve
	$$\beta:[0,1]\rightarrow C_0\subseteq(\R^d\backslash K)\cap (x+W)$$
	with $\beta(0)=\xi_0$, $\beta(1)=v$, and $\beta'(t)\in W$ wherever $\beta$ is differentiable. Again, we set
	$$s:=\inf\{t>0;\,\beta(t)\in\partial X_2\}$$
	so that again $s>0$ and again we define
	$$\tilde{\beta}:[0,\infty)\rightarrow (X_2\backslash K)\cap (x+W),\tilde{\beta}(t):=\beta(\frac{t}{1+t}s).$$
	Then, again
	$$\alpha:[0,\infty)\rightarrow(X_2\backslash K)\cap(x+W),\alpha(t):=\begin{cases}
		\tilde{\gamma}(t),&t\leq T+\lambda_0\\ \tilde{\beta}(t-T-\lambda_0),&t>T+\lambda_0,
	\end{cases}$$
	fulfills all desired properties in the last case that remained which finally proves the theorem.
\end{proof}

\section{Some auxiliary results on the Cauchy problem}\label{prelim}

The purpose of this section is to collect some results on the non-characteristic Cauchy problem which will be used in the proof of Theorem \ref{zero solution} in section \ref{proof zero solution} below. We assume that these results are known but since we could not find any reference with a correct representation of the solution of the non-characteristic Cauchy problem given in Remark \ref{explicit formula} below, we include its derivation here for the reader's convenience.

Let $P$ be such that $e_d\in\R^d$ is not characteristic for $P$ and we write $P(x_1,\ldots,x_d)=\sum_{k=0}^m Q_k(x_1,\ldots,x_{d-1})x_d^k$ with $Q_k\in\C[X_1,\ldots,X_{d-1}],0\le k\leq m$, where the degree of $Q_k$ is bounded by $m-k$. Since $e_d$ is non-characteristic for $P$ and since $P$ is of degree $m$, we have $Q_m=c\in\C\backslash\{0\}$ and we assume without loss of generality that $Q_m=1$.\footnote{Note added in proof: While in view of Theorem \ref{zero solution} we are interested in the non-characteristic Cauchy problem, it should be noted that the results of the current section hold true for polynomials $P$, not necessarily of degree $m$ and for which $e_d$ need not be a non-characteristic vector, as long as $P$ can be written as $P(x_1,\ldots,x_d)=\sum_{k=0}^m Q_k(x_1,\ldots,x_{d-1})x_d^k$ with $Q_k\in\mathbb{C}[X_1,\ldots,X_{d-1}]$ and $Q_m=c\in\mathbb{C}\backslash \{0\}$ and $m\in\N$.} 

As mentioned at the end of section \ref{statement of the main results}, we will achieve our objective to construct a zero solution for $P(D)$ with support bounded with respect to the $x_1$-axis by explicitly solving a certain Cauchy problem for $P(D)$ with Cauchy data on the non-characteristic hyperplane $\{x\in\R^d;\,x_d=0\}$. In order to formulate the solution in a convenient way we introduce the following notion.

\begin{definition}
	For $Y\subseteq\R^{d-1}$ open we define recursively for $l\in \N_0$
	$$\mathscr{C}_l:\mathscr{E}(Y)\rightarrow \mathscr{E}(Y), f\mapsto\begin{cases}
		0, &l\in\{0,\ldots,m-2\},\\
		f, &l=m-1,\\
		-\sum_{k=0}^{m-1}Q_k(D)\mathscr{C}_{k+l-m}(f), &l\geq m.
	\end{cases}$$
	Thus, by definition, $\mathscr{C}_{m+l}(f)+\sum_{k=0}^{m-1}Q_k(D)\mathscr{C}_{k+l}(f)=0$ for all $l\in\N_0, f\in \mathscr{E}(Y)$.
\end{definition} 

A straight-forward calculation yields the next result.

\begin{proposition}\label{preparation boundary value}
	Let $Y\subseteq\R^{d-1}$ be open and $h_0,\ldots,h_{m-1}\in \mathscr{E}(Y)$. Then
	$$\forall\,0\leq s\leq m-1:\,\sum_{j=0}^s\sum_{k=m-1-s}^{m-1-j}Q_{j+k+1}(D)\mathscr{C}_{k+s}(h_j)=h_s.$$
\end{proposition}


\begin{definition}
	Let $Y\subseteq\R^{d-1}$ be open. We define
	$$L_n:\mathscr{E}(Y)\rightarrow \mathscr{E}(Y\times\R), L_n(h)(x,x_d):=\sum_{l=0}^n\mathscr{C}_{l}(h)(x)\frac{(i\,x_d)^l}{l!}.$$
	Obviously, $L_n$ is a linear and continuous mapping.
\end{definition}

Keeping in mind that $Q_m=1$ and that $\mathscr{C}_{m+l}(h_j)+\sum_{r=0}^{m-1}Q_r(D)\mathscr{C}_{r+l}(h_j)=0$ for all $l\in\N_0$, with the aid of Proposition \ref{preparation boundary value} one easily derives the next proposition.

\begin{proposition}\label{solution Cauchy}
	Let $Y\subseteq\R^{d-1}$ be open, $h_0,\ldots,h_{m-1}\in \mathscr{E}(Y)$ be such that for all $0\leq j\leq m-1$ the sequence $(L_n(h_j))_n$ converges in $\mathscr{E}(Y\times\R)$, $L(h_j):=\lim_{n\rightarrow\infty}L_n(h_j)=\sum_{l=0}^\infty\mathscr{C}_{l}(h_j)\frac{(ix_d)^l}{l!}$.
	
	Then, $u:=\sum_{j=0}^{m-1}\sum_{k=0}^{m-1-j}Q_{j+k+1}(D)D_d^kL(h_j)\in \mathscr{E}_P(Y\times R)$ and $D_d^su(\cdot,0)=h_s(\cdot), 0\leq s\leq m-1.$
\end{proposition}



In order to determine when the sequence $(L_n(h_j))_n$ converges in $\mathscr{E}(Y\times\R)$ we next give an explicit representation of the recursively defined operators $\mathscr{C}_{l}$.

\begin{proposition}\label{prop:explicit}
	Let $Y\subseteq\R^{d-1}$ be open. For each $l\in\N_0$ and every $f\in \mathscr{E}(Y)$ it holds
	$$\mathscr{C}_{m-1+l}(f)=\sum_{s\in\N_0^m, \sigma(s)=l}(-1)^{|s|}\binom{|s|}{s_1,\ldots,s_m}\prod_{k=1}^m Q_{m-k}^{s_k}(D)f,$$
	where $\sigma(s)=\sum_{j=1}^m js_j$.
\end{proposition}

\begin{proof}
	We prove the claim by induction on $l$. For $l=0$ the claim holds true since $\mathscr{C}_{m-1}(f)=f$.
	
	Next, we assume the claim to be true for all $n\leq l$. For $f\in \mathscr{E}(Y)$ we have
	$$\mathscr{C}_{m-1+(l+1)}(f)=-\sum_{k=0}^{m-1}Q_k(D)\mathscr{C}_{k+l}(f)=-\sum_{s=l}^{m-1+l}Q_{s-l}(D)\mathscr{C}_{s}(f).$$
	Now we have to distinguish two cases. As a first case we consider $l\leq m-1$. Since $\mathscr{C}_{r}(f)=0$ whenever $r<m-1$, we continue
	\begin{align}\label{induction1}
		=&-\sum_{s=l}^{m-1+l}Q_{s-l}(D)\mathscr{C}_{s}(f)=-\sum_{s=m-1}^{m-1+l}Q_{s-l}(D)\mathscr{C}_{s}(f)\nonumber\\
		=&-\sum_{n=0}^l Q_{m-1-l+n}(D)\mathscr{C}_{m-1+n}(f)\\
		=&-\sum_{n=0}^{l}Q_{m-1-l+n}(D)\Big(\sum\limits_{\substack{s\in\N_0^m,\\ \sigma(s)=n}}(-1)^{|s|}\binom{|s|}{s_1,\ldots,s_m}\prod_{k=1}^m Q_{m-k}^{s_k}(D)f\Big),\nonumber
	\end{align}
	where we have used the induction hypothesis in the last step. Now, writing $\pi_j(s):=s_j$ for $s\in\N_0^m$, we observe that $$\sigma(s)=\sum_{j=1}^mjs_j=\sum_{j=1}^m j\pi_j(s+e_{l+1-n})-(l+1-n)$$
	so that
	$$\forall\,s\in\N_0^m, 0\leq n\leq l:\,\sigma(s)=n\Leftrightarrow\sigma(s+e_{l+1-n})=l+1.$$
	Continuing with the calculation (\ref{induction1}) we obtain
	\begin{align*}
		=&-\sum_{n=0}^{l}Q_{m-1-l+n}(D)\Big(\sum_{s\in\N_0^m,\sigma(s)=n}(-1)^{|s|}\binom{|s|}{s_1,\ldots,s_m}\prod_{k=1}^m Q_{m-k}^{s_k}(D)f\Big)\\
		=&-\sum_{n=0}^l\sum\limits_{\substack{s\in\N_0^m,\\ \sigma(s)=n}}(-1)^{|s|}\binom{|s|}{s_1,\ldots,s_m} \prod_{k=1}^m  Q_{m-k}^{\pi_k(s+e_{l+1-n})}(D)f\\
		=&\sum_{n=0}^l\sum\limits_{\substack{s\in\N_0^m,\\ \sigma(s+e_{l+1-n})=l+1}}(-1)^{|s+e_{l+1-n}|}\binom{|s|}{s_1,\ldots,s_m} \prod_{k=1}^m  Q_{m-k}^{\pi_k(s+e_{l+1-n})}(D)f\\
		=&\Big(\sum_{r=1}^{l+1}\sum\limits_{\substack{s\in\N_0^m,\\ \sigma(s+e_r)=l+1}}(-1)^{|s+e_r|}\binom{|s|}{s_1,\ldots,s_m}\prod_{k=0}^{m-1}Q_{m-k}^{\pi_k(s+e_r)}(D)
		\Big)f\\
		=&\Big(\sum\limits_{\substack{t\in\N_0^m,\\ \sigma(t)=l+1}}(-1)^{|t|}\Big(\sum_{r=1}^{l+1}\binom{|t|-1}{t_1,\ldots,t_{r-1},t_r-1,t_{r+1},\ldots,t_{l+1},0\ldots,0}\Big)\prod_{k=1}^m Q_{m-k}^{t_k}(D)\Big)f,
	\end{align*} 
	where in the last step we only rearranged the summands with respect to those $s\in\N_0^m$ for which $s+e_r, 1\leq r\leq l+1,$ gives the same $t\in\N_0^m$. Summarizing, in case of $l\leq m-1$ we obtain
	\begin{align*}
		&\mathscr{C}_{m-1+(l+1)}(f)\\
		=&\Big(\sum\limits_{\substack{t\in\N_0^m,\\ \sigma(t)=l+1}}(-1)^{|t|}\Big(\sum_{r=1}^{l+1}\binom{|t|-1}{t_1,\ldots,t_{r-1},t_r-1,t_{r+1},\ldots,t_{l+1},0\ldots,0}\Big)\prod_{k=1}^m Q_{m-k}^{t_k}(D)\Big)f.
	\end{align*}
	In case of $l>m-1$ we have - using the induction hypothesis in the third step
	\begin{align*}
		&\mathscr{C}_{m-1+(l+1)}(f)=-\sum_{s=l}^{m-1+l}Q_{s-l}(D)\mathscr{C}_{s}(f)\\
		=&-\sum_{n=0}^{m-1}Q_n(D)\mathscr{C}_{l+n}(f)=-\sum_{n=0}^{m-1}Q_n(D)\mathscr{C}_{m-1+(l-m+1)+n}(f)\\
		=&-\sum_{n=0}^{m-1}Q_n(D)\Big(\sum\limits_{\substack{s\in\N_0^m,\\ \sigma(s)=l-m+1+n}}(-1)^{|s|}\binom{|s|}{s_1,\ldots,s_m}\prod_{k=1}^m Q_{m-k}^{s_k}(D)\Big)f\\
		=&\Big(\sum_{n=0}^{m-1}\sum\limits_{\substack{s\in\N_0^m,\\ \sigma(s)=l+1-(m-n)}}(-1)^{|s+e_{m-n}|}\binom{|s|}{s_1,\ldots,s_m}\prod_{k=1}^m Q_{m-k}^{\pi_k(s+e_{m-n})}(D)\Big)f\\
		=&\Big(\sum_{n=0}^{m-1}\sum\limits_{\substack{s\in\N_0^m,\\ \sigma(s+e_{m-n})=l+1}}(-1)^{|s+e_{m-n}|}\binom{|s|}{s_1,\ldots,s_m}\prod_{k=1}^m Q_{m-k}^{\pi_k(s+e_{m-n})}(D)\Big)f\\
		=&\Big(\sum_{r=1}^m \sum\limits_{\substack{s\in\N_0^m,\\ \sigma(s+e_r)=l+1}}(-1)^{|s+e_r|}\binom{|s|}{s_1,\ldots,s_m}\prod_{k=1}^m Q_{m-k}^{\pi_k(s+e_r)}(D)\Big)f\\
		=&\Big(\sum\limits_{\substack{t\in\N_0^m,\\ \sigma(t)=l+1}}(-1)^{|t|}\Big(\sum_{r=1}^m\binom{|t|-1}{t_1,\ldots,t_{r-1},t_r-1,t_{r+1},\ldots,t_m}\Big)\prod_{k=1}^m Q_{m-k}^{t_k}(D)\Big)f.
	\end{align*}
	Summarizing, whether $l\leq m-1$ or $l>m-1$, we obtain
	$$\mathscr{C}_{m-1+(l+1)}(f)=\sum\limits_{\substack{t\in\N_0^m,\\ \sigma(t)=l+1}}(-1)^{|t|}\alpha(t)\prod_{k=1}^m Q_{m-k}^{t_k}(D)f,$$
	where
	$$\alpha(t)=\begin{cases}
		l\leq m-1:&\sum_{r=1}^{l+1}\binom{|t|-1}{t_1,\ldots,t_{r-1},t_r-1,t_{r+1},\ldots,t_{l+1},0\ldots,0},\\
		l>m-1:&\sum_{r=1}^m\binom{|t|-1}{t_1,\ldots,t_{r-1},t_r-1,t_{r+1},\ldots,t_m}.
	\end{cases}$$
	Since the function
	$$M_k:\R^k\rightarrow\R, M_k(x)=\begin{cases} 0,& x\notin\N_0^k\\ \binom{|x|}{x_1,\ldots,x_k},&x\in\N_0^k\end{cases}$$
	satisfies $\sum_{r=1}^k M_k(x-e_r)=M_k(x)$ it follows
	\begin{align*}
		\sum_{r=1}^{l+1}\binom{|t|-1}{t_1,\ldots,t_{r-1},t_r-1,t_{r+1},\ldots,t_{l+1},0\ldots,0}&=\sum_{r=1}^{l+1}M_{l+1}((t_1,\ldots,t_{l+1})-e_r)\\
		&=\binom{|t|}{t_1,\ldots,t_{l+1},0,\ldots,0}
	\end{align*}
	as well as
	$$\sum_{r=1}^m\binom{|t|-1}{t_1,\ldots,t_{r-1},t_r-1,t_{r+1},\ldots,t_m}=\sum_{r=1}^m M_m((t_1,\ldots,t_m)-e_r)=\binom{|t|}{t_1,\ldots,t_m}.$$
	Taking into account that $t_{l+2}=\ldots=t_m=0$ whenever $\sigma(t)=l+1\leq m$, we finally arrive at
	$$\mathscr{C}_{m-1+(l+1)}(f)=\sum\limits_{\substack{t\in\N_0^m,\\ \sigma(t)=l+1}}(-1)^{|t|}\binom{|t|}{t_1,\ldots,t_m}\prod_{k=1}^m Q_{m-k}^{t_k}(D)f,$$
	which proves the claim for $l+1$. The proof is complete.
\end{proof}

\begin{rmk}\label{explicit formula}
	Using the explicit formula for $\mathscr{C}_{m-1+l}, l\in\N_0,$ it follows for $n\geq m-1, h\in \mathscr{E}(Y)$, and $(x,x_d)\in Y\times \R$
	$$L_n(h)(x,x_d)=\sum_{l=m-1}^n\sum\limits_{\substack{s\in\N_0^m,\\ \sigma(s)=l-m+1}}(-1)^{|s|}\binom{|s|}{s_1,\ldots,s_m}\prod_{k=1}^m Q_{m-k}^{s_k}(D)h(x) \frac{(ix_d)^l}{l!},$$
	where $\sigma(s)=\sum_{j=1}^m js_j$.
	
	Thus, if $h_0,\ldots,h_{m-1}\in\mathscr{E}(Y)$ are such that $(L_n(h_j))_{n\in\N}$ converge in $\mathscr{E}(Y\times\R)$ the solution to the Cauchy problem $P(D)u=0$ in $Y\times\R$, $D_d^ju(\cdot,0)=h_j, 0\leq j\leq m-1$ is given by
	$$u(x,x_d)=\sum_{j=0}^{m-1}\sum_{l=m-1}^\infty\sum\limits_{\substack{s\in\N_0^m,\\ \sigma(s)=l-m+1}}(-1)^{|s|}\binom{|s|}{s_1,\ldots,s_m}\prod_{k=1}^m Q_{m-k}^{s_k}(D)h_j(x) \frac{(ix_d)^l}{l!}.$$
\end{rmk}

\section{Proof of Theorem \ref{zero solution}}\label{proof zero solution}

As mentioned at the end of Section \ref{statement of the main results}, for the proof we combine a result of H\"ormander \cite[Theorem 8.6.7]{HoermanderPDO1} with an idea of Langenbruch from \cite{Langenbruch1983}. By H\"ormander's result, for $P$ with characteristic vector $N$ there is $u\in \mathscr{E}(\R^d)$ with $P(D)u=0$ and $\supp u=\{x\in\R^d;\langle x,N\rangle\leq 0\}$. The power series approach to the solution of the (non-characteristic) Cauchy problem from the previous section was used by Langenbruch to construct, for a certain class of polynomials $P$, a fundamental solution which has a bounded support with respect to some of the variables. The proof of Theorem \ref{zero solution} will be achieved by combining both results/ideas.\\

\begin{proof}[Proof of Theorem \ref{zero solution}]
	By hypothesis on the polynomial $P$, there is $\gamma\in (0,1)$ such that $\deg_{x_1}(Q_k)\leq \gamma (m-k)$ for all $0\leq k\leq m-1$. We fix $a,\varepsilon>0$ with $a>\varepsilon$. Next we fix $\rho\in (1,\frac{1}{\gamma})$ and we denote by $\Gamma^{(\rho)}(\R)$ the Gevrey class of order $\rho$, i.e.\ $\Gamma^{(\rho)}(\R)$ consists of those smooth functions $f$ on $\R$ for which for every compact $L\subset\R$ there are constants $C, R>0$ such that
	$$\forall\,x\in L, \alpha\in\N_0:\,|f^{(\alpha)}(x)|\leq C\,R^\alpha \alpha^{\rho\alpha}.$$
	Since $\rho>1$ there is $g\in\Gamma^{(\rho)}(\R)\cap\mathscr{D}(\R)$ such that $\supp g\subseteq [-(a+\varepsilon),-\frac{\varepsilon}{2}]$ and $g=1$ in a neighborhood of $[-(a+\frac{\varepsilon}{2}),-\frac{3\varepsilon}{4}]$. The existence of $g$ follows for example from an application of \cite[Theorem 1.4.2]{HoermanderPDO1}.
	
	By \cite[Proof of Theorem 8.6.7]{HoermanderPDO1} for each characteristic vector $N$ of $P$ and each non-characteristic vector $\xi$ for $P$ there is a Puiseux series $t(s)=s\sum_{j=1}^\infty c_j(s^{-1/p})^j$, analytic for $s\in\C, |s^{1/p}|>M$ for suitable $M>0$, such that for $\tau>(2M)^p$ and $1-1/p<r<1$
	$$v(x)=\int_{i\tau-\infty}^{i\tau+\infty}e^{i\langle x,sN+t(s)\xi\rangle}e^{-(s/i)^r}ds$$
	is a smooth function on $\R^d$ with $P(D)v=0$, $\supp v=\{x\in\R^d;\,\langle x,N\rangle\leq 0\}$, and such that $v_{|\{x\in\R^d;\,\langle x,N\rangle\neq 0\}}$ is real analytic, where $(s/i)^r$ is defined so that it is real and positive when $s$ is on the positive imaginary axis. Moreover, the definition of $v$ is independent of the particular choice of $\tau>(2M)^p$ and
	$$\forall\,\alpha\in\N_0^d:\,\partial^\alpha v(x)=\int_{i\tau-\infty}^{i\tau+\infty}(sN+t(s)\xi)^\alpha e^{i\langle x,sN+t(s)\xi\rangle}e^{-(s/i)^r}ds.$$
	Since by hypothesis $e_1$ is characteristic for $P$ while $e_d$ is not, in the above definition of $v$ we can choose $N=e_1$ and $\xi=e_d$ yielding a smooth function $v$ on $\R^d$ which only depends on $x_1$ and $x_d$ such that $P(D)v=0$, $\supp v=\{x\in\R^d;\,x_1\leq 0\}$, and $v_{|\{x\in\R^d;\,x_1\neq 0\}}$ is real analytic. Since real analytic functions belong to $\Gamma^{(\rho)}$ and since $\Gamma^{(\rho)}$ is an algebra which is closed under differentiation, it follows that
	$$\forall\,0\leq j\leq m-1:\,h_j:\R^{d-1}\rightarrow\C, x'\mapsto g(x'_1)D_d^j v(x',0)$$
	belong to $\Gamma^{(\rho)}(\R^{d-1})$, depend only on $x'_1$ and satisfy
	\begin{equation}\label{supports}
		\supp h_j=\{x'\in\R^{d-1};\,x'_1\in\supp g\}\subseteq\{x'\in\R^{d-1};x'_1\in[-(a+\varepsilon),-\varepsilon/2]\}.
	\end{equation}
	In particular, there are $C>0,R\geq 1$ such that
	\begin{equation}\label{useful estimate}
		\forall\,0\leq j\leq m-1\,\forall\,x'\in\R^{d-1},\alpha\in\N_0^{d-1}:\,|D^\alpha h_j(x')|\leq C R^{\alpha_1} \alpha_1^{\rho\alpha_1}.
	\end{equation}
	For every $0\leq k\leq m$ we have $Q_k(x')=\sum_{|\alpha'|\leq m-k}q_{k,\alpha'}x'^{\alpha'}$ for suitable $q_{k,\alpha'}\in\C$. We fix $q>0$ such that $\sum_{|\alpha'|\leq m-k}|q_{k,\alpha'}|\leq q$ for all $k,\alpha'$. Moreover, we observe that for every $s\in\N_0^m$ due to the hypothesis on $\deg_{x_1}(Q_k)$
	$$\deg_{x_1}\Big(\prod_{k=1}^m Q_{m-k}^{s_k}\Big)=\sum_{k=1}^m s_k\deg_{x_1}(Q_{m-k})\leq \gamma\sum_{k=1}^m k s_k=\gamma\,\sigma(s).$$
	Applying (\ref{useful estimate}) it follows that for every $0\leq j\leq m-1$ and each $s\in\N_0^m$
	$$\forall\,x'\in\R^{d-1}:\,\left|\left(\prod_{k=1}^m Q_{m-k}^{s_k}(D)\right)h_j(x')\right|\leq  q^{|s|} C R^{\gamma\sigma(s)}\big(\gamma\sigma(s)\big)^{\rho\gamma\,\sigma(s)}.$$
	Thus, for $B>0$ it follows from Remark \ref{explicit formula} that for every $n\geq m-1, k\in\N,$ and each $x'\in \R^{d-1}$, $x_d\in\R$ with $|x_d|\leq B$ we have by an application of the Multinomial Theorem and Stirling's Formula for a suitable constant $\tilde{C}$
	\begin{align*}
		&|L_{n+k}(h_j)(x',x_d)-L_n(h_j)(x',x_d)|\\
		&\leq \sum_{l=n+1}^{n+k}\sum\limits_{\substack{s\in\N_0^m,\\ \sigma(s)=l-m+1}}\binom{|s|}{s_1,\ldots,s_m}q^{|s|}CR^{\gamma\sigma(s)}\big(\gamma\sigma(s)\big)^{\rho\gamma\,\sigma(s)}\frac{B^l}{l!}\\
		&\leq \tilde{C}\sum_{l=n+1}^{n+k}(m q)^{l-m+1} R^{\gamma(l-m+1)} \big(\gamma(l-m+1)\big)^{\rho\gamma(l-m+1)}\frac{B^l}{l^l}\\
		&<\tilde{C}\sum_{l=n+1}^\infty\Big(\frac{m q R B}{l^{1-\rho\gamma}}\Big)^l<\infty,
	\end{align*}
	because $\rho\gamma<1$. Thus $(L_n(h_j))_{n\in\N}, 0\leq j\leq m-1,$ converge uniformly on $\R^{d-1}\times [-B,B]$. It is at this point where we need that $h_j\in \Gamma^{(\rho)}(\R^{d-1})$; for this to hold, we have chosen $g\in \Gamma^{(\rho)}(\R)$ since for an arbitrary cut-off function in place of $g$ it needs not be true that $h_j\in\Gamma^{(\rho)}(\R^{d-1})$.
	
	The explicit formula for $L_n(h_j)$ in Remark \ref{explicit formula} shows that for all $\alpha\in\N_0^d$ and each $(x',x_d)\in\R^d$ we have for $n\geq m-1$
	\begin{align*}
		&\partial^\alpha L_n(h_j)(x,x_d)\\
		=&\sum_{l=\max\{m-1,\alpha_d\}}^n\sum\limits_{\substack{s\in\N_0^m,\\ \sigma(s)=l-m+1}}(-1)^{|s|}\binom{|s|}{s_1,\ldots,s_m}\prod_{k=1}^m Q_{m-k}^{s_k}(D)\partial^{\alpha'}h_j(x')\frac{(ix_d)^{l-\alpha_d}}{(l-\alpha_d)!}.
	\end{align*}
	Since $\Gamma^{(\rho)}(\R^{d-1})$ is closed under differentiation, similar estimates to the ones elaborated above show that $(\partial^\alpha L_n(h_j))_{n\in\N}$ converges uniformly in $\R^{d-1}\times[-B,B]$ for every $B>0$ so that $(L_n(h_j))_{n\in\N}$ converges in $\mathscr{E}(\R^d), 0\leq j\leq m-1$. Denoting the respective limits by $L(h_j), 0\leq j\leq m-1,$ it follows from (\ref{supports}) and the explicit formula for $L_n(h_j)$ in Remark \ref{explicit formula} that for each $0\leq j\leq m-1$
	\begin{equation}\label{supports2}
		\supp L(h_j)\subseteq\{x\in\R^d;\,x_1\in[-(a+\varepsilon),-\varepsilon/2]\}.
	\end{equation}
	
	Proposition \ref{solution Cauchy} implies that the smooth function $u$ on $\R^d$ defined as
	$$u:=\sum_{j=0}^{m-1}\sum_{k=0}^{m-1-j}Q_{j+k+1}(D)D_d^{k}L(h_j)$$
	satisfies $P(D)u=0$ and $D_d^j u(x',0)=g(x'_1)D_d^j v(x',0)$ for every $0\leq j\leq m-1$. Moreover, by (\ref{supports2})
	$$\supp u \subseteq\{x\in\R^d;\,x_1\in[-(a+\varepsilon),-\varepsilon/2]\}.$$
	Since $g=1$ in a neighborhood of $[-(a+\varepsilon/2),-3\varepsilon/4]$ and since $D_d^j v(\cdot,0)$ are real analytic it follows from Holmgren's Uniqueness Theorem (see e.g.\ \cite[Section V.5.3]{HoermanderPDO} or \cite{Mizohata1973}) and the fact that $\{x\in \R^d;\,x_d=0\}$ is a non-characteristic hyperplane for $P$ that $u$ and $v$ coincide on the set $\{x\in\R^d;\,-a\leq x_1\leq -\varepsilon\}$. Since the latter set is contained in the support of $v$, the theorem is proved. 
\end{proof}

\section{Proofs of Theorem \ref{necessary condition} and Corollaries \ref{global approximation} and \ref{non-degenerate parabolic}}\label{proof necessity}

In this section we finally prove Theorem \ref{necessary condition} and Corollaries \ref{global approximation} and \ref{non-degenerate parabolic}.

\begin{proof}[Proof of Theorem \ref{necessary condition}]
	The proof will be done by contradiction. Thus, we assume that there is $c\in \R$ such that $X_2$ contains a compact connected component $C$ of $(\R^d\backslash X_1)\cap H_c$.
	
	Let $V\subseteq U\subseteq\R^{d-1}$ be open and bounded such that
	$$C\subseteq \{c\}\times V\subseteq\{c\}\times\overline{V}\subseteq\{c\}\times U\subseteq\{c\}\times\overline{U}\subseteq (X_1\cup C)\cap H_c.$$
	Then, $\{c\}\times(\overline{U}\backslash V)$ is a compact subset of $X_1\cap H_c$ and therefore $$\delta:=\dist(\{c\}\times(\overline{U}\backslash V),\R^d\backslash X_1)>0.$$
	By compactness and $X_1\cup C\subseteq X_2$ there is $\varepsilon>0$ such that
	\begin{enumerate}
		\item[i)] $[c-\varepsilon,c+\varepsilon]\times (\overline{U}\backslash V)\subseteq X_1$,
		\item[ii)] $[c-\varepsilon,c+\varepsilon]\times\overline{U}\subseteq X_2$.
	\end{enumerate}
	Applying Theorem \ref{zero solution} to the polynomial $\check{P}(x):=P(-x)$ there is $u\in \mathscr{E}_{\check{P}}(\R^d)$ with
	$$\left\{x\in\R^d;-\frac{3\varepsilon}{2}\leq x_1\leq -\frac{\varepsilon}{2}\right\}\subseteq\supp u\subseteq\{-2\varepsilon\leq x_1\leq 0\}$$
	such that $u$ is real analytic in $(-3\varepsilon/2,-\varepsilon/2)\times\R^{d-1}$. Thus, $v(x):=u(x_1-c-\varepsilon, x_2,\ldots,x_d)$ defines a smooth function on $\R^d$ satisfying $\check{P}(D)v=0$ and
	$$\left\{x\in\R^d;\,c-\frac{\varepsilon}{2}\leq x_1\leq c+\frac{\varepsilon}{2}\right\}\subseteq\supp v\subseteq\{x\in\R^d;\,c-\varepsilon\leq x_1\leq c+\varepsilon\}.$$
	Moreover, the restriction of $v$ to $(c-\varepsilon/2,c+\varepsilon/2)\times\R^{d-1}$ is real analytic.
	
	Next, we choose $\psi\in\mathscr{D}(\R^{d-1})$ with $\supp\psi\subseteq U$ and $\psi=1$ in a neighborhood of $\overline{V}$, and set $w(x):=v(x)\psi(x_2,\ldots,x_d)$. Then
	$$\supp w\subseteq\supp v\cap (\R\times\supp \psi)\subseteq [c-\varepsilon, c+\varepsilon]\times U$$
	as well as
	$$\supp\check{P}(D)w\subseteq\supp v\cap (\R\times \supp(d\psi))\subseteq[c-\varepsilon,c+\varepsilon]\times(U\backslash \overline{V}).$$
	In particular, $w\in\mathscr{D}(X_2)$ and $\check{P}(D)w\in\mathscr{D}(X_1)$. For each $f\in \mathscr{D}'_P(X_2)$ we have
	$$\langle f,\check{P}(D)w\rangle=\langle P(D)f,w\rangle =0.$$
	On the other hand, for arbitrary fixed $x_0\in C\subseteq\{c\}\times V$, since $v$ is real analytic in $(c-\varepsilon/2,c+\varepsilon/2)\times\R^{d-1}$ and the latter set is contained in the support of $v$, there is $\alpha_0\in\N_0^d$ such that $\partial^{\alpha_0}v(x_0)\neq 0$. Let $E$ be a fundamental solution for $P(D)$ then $\partial^{\alpha_0}\tau_{x_0}E_{|X_1}\in \mathscr{D}'_P(X_1)$, where $\tau_{x_0}$ denotes translation by $x_0$. Since $\check{P}(D)w\in\mathscr{D}(X_1)$ we conclude
	$$\langle\check{P}(D)w,\partial^{\alpha_0}\tau_{x_0}E_{|X_1}\rangle=\langle w,\partial^{\alpha_0}\delta_{x_0}\rangle=(-1)^{|\alpha_0|}\partial^{\alpha_0}v(x_0)\neq 0.$$
	Thus $\check{P}(D)w$ is a non-trivial continuous linear functional on $\mathscr{D}'_P(X_1)$ which vanishes on $r_{\mathscr{D}'}\Big(\mathscr{D}'_P(X_2)\Big)$ so by the Hahn-Banach Theorem the latter subspace of $\mathscr{D}'_P(X_1)$ is not dense giving the desired contradiction.
\end{proof}

\begin{proof}[Proof of Corollary \ref{global approximation}]
	By Theorem \ref{necessary condition}, i) implies ii). To prove the converse implication as well as the statement that $(X,\R^d)$ is a $P$-Runge pair, we first note that $\R^d$ is $P$-convex for supports. Thus, in view of Theorem \ref{sufficiency for P-Runge pairs}, it remains to show that $X$ is $P$-convex for supports, whenever ii) holds. Assuming that $X$ is not $P$-convex for supports, it follows from Theorem \ref{characterization p-convexity} that there are $c\in\R$ and a compact subset $K\subseteq H_c\cap X$ with
	\begin{equation}\label{contradiction}
		\min_{x\in K}\dist(x,X^c)<\min_{x\in\partial_{H_c}K}\dist(x,X^c),
	\end{equation}
	where as usual $H_c=\{x\in\R^d;\,x_1=c\}$. Let $x^0\in K$ and $y^0\in X^c$ be such that
	\[|x^0-y^0|=\min_{x\in K}\dist(x,X^c).\]
	Since $x^0\in K\subseteq H_c\cap X$ we have $H_c=H_{x_1^0}$,
	\[H_{y_1^0}=(y^0-x^0)+H_{x_1^0}\supseteq (y^0-x^0)+K,\]
	and by (\ref{contradiction})
	\[H_{y_1^0}\cap X\supseteq (y^0-x^0)+\partial_{H_{x_1^0}}K=\partial_{H_{y_1^0}}\big(y^0-x^0+K\big).\]
	The latter implies that the connected component of $(\R^d\backslash X)\cap H_{y_1^0}$ which contains $y^0$ is bounded, hence compact which gives a contradiction.
\end{proof}

It remains to prove Corollary \ref{non-degenerate parabolic}.

\begin{proof}[Proof of Corollary \ref{non-degenerate parabolic}]
	The set of zeros of the principal part of each of the polynomials $P(t,x)=(-i t)^r+Q(x)$ and $P(t,x)=-t-|x|^2$ is the time axis, i.e.\ $t$-axis, and thus the characteristic hyperplanes are the hyperplanes orthogonal to the time axis. Thus, by Theorem \ref{characterization p-convexity}, $I_1\times X_2$ and $I_2\times X_2$ are both $P$-convex for supports. Moreover, $P$ satisfies the additional hypothesis of Theorem \ref{sufficiency for P-Runge pairs}. Thus, $I_1\times X_1$ and $I_2\times X_2$ form a $P$-Runge pair if and only if for every $t_0\in R$ no compact connected component of
	$$\big(\R^{n+1}\backslash (I_1\times X_1)\big)\cap\{(t_0,x);\, x\in\R^n\}$$
	is contained in $I_2\times X_2$. The latter condition is obviously equivalent to $\R^n\backslash X_1$ not having a compact connected component in $X_2$. Thus ii) follows from Theorem \ref{sufficiency for P-Runge pairs} while i) follows from Theorem \ref{sufficiency for P-Runge pairs} once it has been taken into account that $P(t,x)=(-it)^r+Q(x)$ is hypoelliptic (cf.\ \cite[Theorem 11.1.11]{HoermanderPDO2}).
\end{proof}

\noindent \textbf{Acknowledgements.} The author is very grateful to D.\ Vogt for pointing out \cite{Langenbruch1983}.

\bibliographystyle{plain}

\end{document}